\documentclass[12pt]{article}

\usepackage{graphicx}

\usepackage{amsmath, amsthm}

\usepackage{amsfonts}

\usepackage{amssymb}

\usepackage{url}

\usepackage{dsfont} 

\usepackage[mathlines]{lineno}

\usepackage{verbatim}

\usepackage{color}

\definecolor{red}{rgb}{1,0,0}

\usepackage{hyperref}

\usepackage{dsfont}

\usepackage{tikz}

\usepackage{rotating}

\usepackage{tkz-graph}

\usepackage{graphicx}

\usepackage{caption}

\usepackage{subcaption}

\tikzstyle{vertex}=[circle, draw, inner sep=0pt, minimum size=6pt]

\tikzstyle{rvertex}=[circle, red, fill, draw, inner sep=0pt, minimum size=6pt]

\tikzstyle{gvertex}=[circle, green, fill, draw, inner sep=0pt, minimum size=6pt]

\tikzstyle{bvertex}=[circle, blue, fill, draw, inner sep=0pt, minimum size=6pt]

\tikzstyle{Bvertex}=[circle, black, fill, draw, inner sep=0pt, minimum size=6pt]

\newcommand{\vertex}{\node[vertex]}

\newcommand{\bvertex}{\node[bvertex]}

\def\noi{\noindent}

\newcommand{\ppt}{\operatorname{ppt}}

\newcommand{\ZFN}{\operatorname{Z}}

\newtheorem{thm}{Theorem}

\newtheorem{defn}[thm]{Definition}

\newtheorem{prop}[thm]{Proposition}

\newtheorem{cor}[thm]{Corollary}

\newtheorem{lem}[thm]{Lemma}

\newtheorem{conj}[thm]{Conjecture}

\newtheorem{ex}[thm]{Example}

\newtheorem{obs}[thm]{Observation}

\newtheorem{rem}[thm]{Remark}

\newcommand{\bit}{\begin{itemize}}
	
	\newcommand{\eit}{\end{itemize}}

\newcommand{\ben}{\begin{enumerate}}
	
	\newcommand{\een}{\end{enumerate}}

\newcommand{\beq}{\begin{equation}}

\newcommand{\eeq}{\end{equation}}

\newcommand{\bea}{\begin{eqnarray*}}
	
	\newcommand{\eea}{\end{eqnarray*}}

\newcommand{\bpf}{\begin{proof}}
	
	\newcommand{\epf}{\end{proof}}

\title{On the power propagation time of a graph}

\author{Chassidy Bozeman\thanks{Department of Mathematics, Iowa State University, Ames, IA 50011, USA (cbozeman@iastate.edu)}}

\date{}

\begin{document}
	
	\maketitle {}

	
	
	%
	
	
	
	
	
	
	
	%
	
	
	
	
	
	
	

	\abstract {In this paper, we give Nordhaus-Gaddum upper and lower bounds on the sum of the power propagation time of a graph and its complement, and we consider the effects of edge subdivisions and edge contractions on the power propagation time of a graph.  We also study a generalization of power propagation time, known as $k-$power propagation time, by characterizing all simple graphs on $n$ vertices whose $k-$power propagation time is $n-1$ or $n-2$ (for $k\geq 1$) and $n-3$ (for $k\geq 2$). We determine all trees on $n$ vertices  whose power propagation time ($k=1$)  is $n-3$, and give  partial characterizations of graphs whose $k-$power propagation time is equal to 1 (for $k\geq 1$). }
	
	
	\section{Introduction} Phasor Measurement Units (PMUs) are machines used by energy companies to monitor the electric power grid. They are placed at selected electrical nodes (locations at which transmission lines, loads,  and generators are connected) within the system. Due to the high cost of the machines, an extensive amount of research has been devoted to minimizing the number of PMUs needed while maintaining the ability to observe the entire system. In \cite{HHHH}, Haynes et al. studied this problems in terms of graphs. \\\indent An electric power grid is modeled by a graph by letting vertices represent the electrical nodes and edges represent transmissions lines between nodes. The {\em power domination process} is defined as follows \cite{HHHH}: A PMU placed at a vertex measures the voltage and phasor angle at that vertex, at the incident edges, and at the vertices at the endpoints of the incident edges. These vertices and edges are said to be {\em observed}. The rest of the system is observed according to the following propagation rules:
	
	\begin{enumerate}
		
		\item[1.] Any vertex that is incident to an observed edge is observed.
		
		\item[2.] Any edge joining two observed vertices is observed.
		
		\item[3.] If a vertex is incident to a total of $t>1$ edges and if $t-1$ of these edges are observed, then all $t$ of these edges are observed.
		
	\end{enumerate}
	
	Here we give an equivalent formulation of the power domination process using our notation as done in  \cite{FHKY15}. Let $G=(V,E)$ be a simple graph and $v\in V(G)$.
The set of neighbors of $v$  is  denoted $N(v)$. For a set $S$ of vertices, the {\em open neighborhood} of $S$ is given by $N(S)=\cup_{s\in S} N(s)$ and the {\em closed neighborhood} of $S$ is $N[S]:= S\cup N(S)$. Given a set $S\subseteq V(G)$, define the following sets:
	
	\begin{enumerate}
		
		\item[1.] $S^{[0]}=S, S^{[1]}=N[S]$.
		
		\item[2.]For $t\geq 1$, $S^{[t+1]}=S^{[t]}\cup \{w\in V(G)| \hspace{1mm}\exists \hspace{1mm} v\in S^{[t]}, N(v)\setminus S^{[t]}=\{w\}\}$.\end{enumerate}
	
For vertices $w$ and $v$ given in (2) we say $v$ {\em forces} $w$. A set $S$ is said to be a {\em power dominating set} if there exists an $\ell$ such that $S^{[\ell]}=V(G).$ The {\em power domination number} of $G$, denoted $\gamma_P(G)$, is the minimum cardinality over all power dominating sets of $G$. Computing $S^{[1]}$ is the {\em domination step} and the computations of $S^{[t+1]}$ (for $t\geq 1$) are the {\em propagation steps}. The authors of \cite{FHKY15} defined the power propagation time: the {\em power propagation time of $G$ with $S$}, denoted $\ppt(G,S)$, is the smallest $\ell$ such that $S^{[\ell]}=V(G)$. The {\em power propagation time of $G$}, denoted $\ppt(G)$,  is given by \[\ppt(G)=\min\{\ppt(G,S)|S \text{ is a minimum power dominating set} \}.\] A minimum power dominating  set $S$ of a graph $G$ is {\em efficient} if $\ppt(G,S)=\ppt(G)$. \\\indent In Section \ref{NGbound}, we give Nordhaus-Gaddum upper and lower bounds for the sum of the power propagation time of a graph and its complement, and in Section \ref{operations} we study the effects of edge subdivision and edge contraction on power propagation time. In Sections \ref{low}  and \ref{high}, we characterize graphs with low and high $k-$power propagation times, respectively. (Note that by letting $k=1$, we characterize graphs with low and high power propagation times.) \\\indent Power domination is closely related to the well known domination problem in graph theory. A set $S\subseteq V(G)$ is a {\em dominating set} if $N[S]=V(G)$. The {\em domination number} of a graph $G,$ denoted $\gamma(G)$,  is the minimum cardinality over all dominating sets of $G$. Note that each dominating set is a power dominating set, so $\gamma_P(G)\leq \gamma(G)$ \cite{HHHH}.

	\subsection{Zero Forcing} The zero forcing problem from combinatorial matrix theory is also closely related to power domination, and in Sections \ref{NGbound} and \ref{low} we use results from zero forcing theory to prove statements about power domination. {\em Zero forcing} is a game played on a graph using the following {\em color change rule:} Let $B$ be a set of vertices of $G$ that are colored blue with $V\setminus B$ colored white. If $v$ is a blue vertex and $u$ is the only neighbor of $v$ that is colored white, then change the color of $u$ to blue. In this case, we say $u$ forces $v$ and write $u\to v$. For a set $B$ of vertices that are initially colored blue, the set of blue vertices that results from applying the color change rule until no more color changes are possible is the {\em final coloring of $B$.} A set $B$ is said to be a {\em zero forcing set} if the final coloring of $B$ is the entire vertex set $V(G)$. The minimum cardinality over all zero forcing sets of $G$ is the {\em zero forcing number} of $G$, denoted $\ZFN(G)$. The zero forcing  number was first introduced in \cite{AIM08} as an upper bound on the linear algebraic parameter of a graph known as the maximum nullity,  and independently in \cite{physics} to study the control of quantum systems.

	\begin{obs}{\rm \cite{PD2015}}\label{NeighborhoodZFS} {\rm A set $S$ is a power dominating set of $G$ if and only if $N[S]$ is a zero forcing set of $G$. It follows that $N(S)\setminus S$ is a zero forcing set of $G\setminus S$.}
		
	\end{obs}
	
	 The authors of \cite{proptime} introduced the propagation time of a zero forcing set of a graph. Due to the close relationship between zero forcing and power domination, many of the questions studied in this paper were motivated by results of the propagation time of a zero forcing set.
	
		\subsection{More notation and terminology}We use $P_n, C_n,$ and $K_n$ to denote the path, cycle, and complete graph on $n$ vertices, respectively. The notation $K_n-e$ represents the complete graph on $n$ vertices minus an edge, and  $K_{s,t}$ is the complete bipartite graph with bipartition $X,Y$ where $|X|=s$ and $|Y|=t$. The graph $L(s,t)$ is the lollipop graph consisting of a complete graph $K_s$ and a path on $t$ vertices where one endpoint of the path is connected to one vertex of $K_s$ via a bridge. \\\indent Let $G=(V,E)$ be a graph and $e=uv\in E(G)$. The graph resulting from {\em subdividing} the edge $e=uv$, denoted $G_e,$ is obtained from $G$ by adding a new vertex $w$ such that $V(G_e)=V(G)\cup \{w\}$ and $E(G_e)=(E(G)\setminus \{uv\})\cup \{uw, wv\}$. To {\em contract} the edge $e=uv$ is to identify vertices $u$ and $v$ as a single vertex $w$ such that $N(w)=(N(u)\cup N(v))\setminus \{u,v\}$. The graph obtained from $G$ by contracting the edge $e$ is denoted by $G/e$.  \\\indent A {\em spider} or {\em generalized star} is a tree formed from a $K_{1,n}$ (for $n\geq 3$) by subdividing any number of its edges any number of times. We use $sp(i_1,i_2,\ldots, i_n)$ to denote the spider obtained from $K_{1,n}$ by subdividing edge $e_j$ a total of $i_j-1$ times for $1\leq j\leq n$. For $G=sp(i_1,i_2,\ldots, i_n)$ and $v$ the unique vertex in $V(G)$ with degree at least 3, we say that the $n$ paths of $G-v$ are the {\em legs} of $G$. 
	\section{Preliminaries}
In this section, we give preliminary results that will be used throughout the remainder of the paper. In particular, Observation \ref{obs} and Lemma \ref{HHHHminpowerdomwithdegthree} are central. We also determine the power propagation time of several families of graphs.
	
		\begin{obs}\label{obs}{\rm
				
				Let $G$ be a graph on $n$ vertices and $S$ a power dominating set of $G$. Then,
				
				\begin{equation}
				\ppt(G,S)\leq n-|S| \\ \label{pptbound1}
				\end{equation} and
				
				\begin{equation}
				\ppt(G,S)-1\leq n-|N[S]| \\ \label{pptbound2}
				\end{equation}
			}\end{obs}
			\noi This follows from the fact that at least one vertex must be forced at each step.
			
\begin{lem}\label{HHHHminpowerdomwithdegthree} {\rm \cite{HHHH}} Let $G$ be a connected graph with $\Delta(G)\geq 3$. Then there exists a minimum power dominating set $S$ of $G$ such that $\deg(s)\geq 3$ for each $s\in S.$
\end{lem}

\subsection{Power propagation time for families} It is well known and clear that the power domination number of the graphs $P_n, C_n, K_n,$ and the spider ${\rm sp}(i_i,i_2,...,i_n)$ is 1. For $G=K_n$, any one vertex is a power dominating set with power propagation time 1. We now determine the power propagation times of the graphs $P_n, C_n,$ and ${\rm sp}(i_i,i_2,...,i_n)$.

			\begin{prop}\label{proppath}Let $P_n$ be the path on $n$ vertices. Then $\gamma_P(P_n)=1$ {\rm and} $\ppt(P_n)=\left \lfloor \frac{n}{2} \right \rfloor$. 
				
			\end{prop}

			\begin{proof}
				
				Let $G=P_n$. Any one vertex of $G$ is a minimum power dominating set. Label the vertices of $G$ with $v_1,\ldots, v_n$ where $\{v_i, v_{i+1}\}\in E(G)$ for $i\in \{1,\ldots, n-1\}$. For any vertex $v_t$, $\ppt(G, \{v_t\})=\max\{t-1, n-t\}.$  It follows that for $n$ odd, $\ppt(G)\geq \frac{n-1}{2}$, and equality is obtained by choosing the power dominating set to be $\{v_t\}$ where $t=\frac{n+1}{2}.$ For $n$ even $\ppt(G)\geq \frac{n}{2},$ and equality is obtained by choosing the power dominating set $\{v_t\}$ with $t\in \{\frac{n}{2}, \frac{n+1}{2}\}$.
			\end{proof}
			
\noindent The proofs of the next three propositions are similar and omitted. 
			
\begin{prop}\label{propcycle}Let $C_n$ be the cycle on $n$ vertices. Then $\gamma_P(C_n)=1$ {\rm and }  $\ppt(C_n)=\left \lfloor \frac{n}{2} \right \rfloor$. \end{prop}

			
\begin{prop}\label{propspider}Let $G={\rm sp}(i_1, i_2,...,i_n)$ {\rm for some} $n\geq 3$.  Then $\gamma_P(G)=1$ {\rm and }  $\ppt(G)=\max\{i_1,i_2,...,i_n\}$. \end{prop}

\begin{prop} For $s,t\geq 3, \gamma_P(K_{s,t})=2$ and $\ppt(K_{s,t})=1$, for  $s\geq 2$ and $t=2$, $\gamma_P(K_{s,t})=1$ and $\ppt(K_{s,t})=2$, and for $s\geq 1$ and $t=1$, $\gamma_P(K_{s,t})=1$ and $\ppt(K_{s,t})=1.$
\end{prop}



			\section{Nordhaus-Gaddum sum bounds for power propagation time}\label{NGbound}
			In 1956, Nordhaus and Gaddum gave upper and lower bounds on the sum and product of the chromatic number of a graph and its complement. Since then, many similar ``Nordhaus-Gaddum" bounds have been studied for other graph parameters. In particular, the Nordhaus-Gaddum sum lower bound for the zero forcing number of a graph on $n$ vertices was established in \cite{PSDZF}: $n-2\leq \ZFN(G)+\ZFN(\overline{G}).$ In this section we use this result to show that for all graphs on $n$ vertices,  $\ppt(G)+\ppt(\overline{G})\leq n+2$. We also conjecture that $n$ is the least upper bound, and demonstrate an infinite family of graphs with $\ppt(G)+\ppt(\overline{G})= n$ for each $G$ in the family. \\\indent The graph $G=K_n$ demonstrates that the Nordhaus-Gaddum sum lower bound is 1. If we require that both $G$ and its complement have edges, then the graph $G=K_{n,n}$  (for $n\geq 3$) demonstrates that Nordhaus-Gaddum sum lower bound is 2.

			\begin{prop}\label{improvedNGppt} Let $G$ be a graph on $n$ vertices. Then $\ppt(G)+\ppt(\overline{G})\leq n+2$.
				
			\end{prop}

			\begin{proof}If $G$ has no edges, then $\ppt(G)=0$ and $\ppt(\overline{G})=1$ so the claim holds. Suppose $G$ and $\overline{G}$ have an edge. Let $S$ be an efficient power dominating set of $G$. Note that $N[S]$ is a zero forcing set of $G$, but it is not minimum: To see this, consider a fixed $s\in S$ (such that $\deg(s)\geq 1$) and a vertex $v_s\in N(s)$. By removing $v_s$, $N[S]\setminus\{v_s\}$ is also a zero forcing set, so $\ZFN(G)+1\leq |N[S]|$. Similarly, $\ZFN(\overline{G})+1\leq |N[S']|$, where $S'$ is an efficient power dominating set of $\overline{G}$. It follows from inequality (\ref{pptbound2}) that $\ppt(G)+\ppt(\overline{G})\leq 2n-(\ZFN(G)+\ZFN(\overline{G})),$ and since $n-2\leq \ZFN(G)+\ZFN(\overline{G})$  (\cite{PSDZF}), then  $\ppt(G)+\ppt(\overline{G})\leq n+2$.
			\end{proof}

			We have not found a graph with $\ppt(G)+\ppt(\overline{G})=n+1$, or one such that $\ppt(G)+\ppt(\overline{G})=n+2$. We have computationally checked all connected graphs on at most 10 vertices and found several graphs with $\ppt(G)+\ppt(\overline{G})=n$. Evidence suggests that this is the least upper bound for all graphs. The next example gives an infinite family of graphs such that $\ppt(G)+\ppt(\overline{G})=n$ for all graphs in the family.

			\begin{ex}\label{examplen}{\rm Let $G_9$ denote the graph given in the Figure \ref{pptn}. For $n\geq 10$, let $G_n$ be a graph on $n$ vertices constructed from $G_{n-1}$ by adding an $n^{th}$ vertex and adding the edges $\{v_{n-2},v_n\}$ and $\{v_{n-1}, v_n\}.$  Note that the set $V(G_n)\setminus \{v_2,v_3\}$ is not a zero forcing set of $G_n$ (since $N(v_2)=N(v_3), $ $v_2$ and $v_3$ will never be forced). So for every power dominating set $S$ of $G_n$, $N[S]$ must contain either $v_2$ or $v_3$. Also note that the sets $\{v_2\}$ and $\{v_3\}$ are minimum power dominating sets of $G_n$ with $\ppt(G_n, v_2)=\ppt(G_n,v_3)=n-3$. Thus, $\gamma_P(G)=1$. For $6\leq i \leq n$, the set $\{v_i\}$ is not a power dominating set since $v_2, v_3\notin N[\{v_i\}]$. Furthermore, it follows from inspection that the sets $\{v_1\}, \{v_4\}, \text{ and } \{v_5\}$ are not power dominating sets. Thus, $\ppt(G_n)=n-3$. \\\indent Similarly, we show that $\ppt(\overline{G_n})=3$. The sets $\{v_{n-1}\}$ and $\{v_n\}$ are power dominating sets of $\overline{G_n}$ with $\ppt(\overline{G_n}, \{v_{n-1}\})=\ppt(\overline{G_n}, \{v_n\})=3,$ and the sets $\{v_2\}$ and $\{v_3\}$ are power dominating sets with $\ppt(\overline{G_n}, \{v_2\})=\ppt(\overline{G_n}, \{v_3\})=4.$ Since $N(v_2)\setminus \{v_3\}=N(v_3)\setminus \{v_2\}$, the set $V(\overline{G_n})\setminus \{v_2,v_3\}$ is not a zero forcing set of $\overline{G_n}.$ So for each power dominating set $S'$ of $\overline{G_n},$ $N[S']$ must contain $v_2$ or $v_3$. It follows that for $i\in \{1,4,5\}$, the set $\{v_i\}$ is not a power dominating set since $v_2, v_3 \notin N[\{v_i\}].$ We now show that for $6\leq i\leq n-2,\{v_i\}$ is not a power dominating set by showing that $N[\{v_i\}]$ is not a zero forcing set. Note that $N[\{v_i\}]=V(\overline{G_n})\setminus\{v_{i-2}, v_{i-1}, v_{i+1}, v_{i+2}\}.$ If $j<i-2, v_j$ is adjacent to $v_{i+1}$ and $v_{i+2}$ (since $v_j$ is not adjacent to $v_{i+1}$ and $v_{i+2}$ in $G_n$). If $j>i+2, v_j$ is adjacent to $v_{i-2}$ and $v_{i-1}$. Thus, no vertex in $N[\{v_i\}]$ is able to perform a force, so $N[\{v_i\}]$ is not a zero forcing set. This shows that $\ppt(\overline{G_n})=3$, so $\ppt(G_n)+\ppt(\overline{G_n})=n$.

					\begin{figure}[h!]
						\begin{center}
							\begin{tikzpicture}[scale=1.0]
							
							
							\vertex [label=above:$v_1$] (1) at (-1,0) {};
							\vertex [label=above:$v_3$] (3) at (0,0) {};
							\vertex [label=above:$v_5$] (5) at (1,0) {};
							\vertex [label=above:$v_7$](7) at (2,0) {};
							\vertex [label=above:$v_9$](9) at (3,0) {};
							\vertex (2)[label=below:$v_2$] at (-0.5,-1) {};
							\vertex (4)[label=below:$v_4$] at (0.5,-1) {};
							\vertex (6)[label=below:$v_6$] at (1.5,-1) {};
							\vertex (8)[label=below:$v_8$] at (2.5,-1) {};
							
							\draw (1) to (3);
							\draw (1) to (2);
							\draw (2) to (5);
							\draw (3) to (4);
							\draw (3) to (5);
							\draw (5) to (7);
							\draw (7) to (9);
							\draw (2) to (4);
							\draw (4) to (6);
							\draw (6) to (8);
							\draw (8) to (9);
							\draw (5) to (6);
							\draw (6) to (7);
							\draw (7) to (8);
							\draw (8) to (9);
							

							\vertex [label=above:$v_9$] (19) at (5,1) {};
							\vertex [label=left:$v_8$] (18) at (4.5,0) {};
							\vertex [label=left:$v_7$] (17) at (4.5 ,-1) {};
							\vertex [label=left:$v_6$](16) at (5,-2) {};
							\vertex [label=below:$v_5$](15) at (6.2,-2.3) {};
							\vertex (14)[label=below:$v_4$] at (7.2,-1.6) {};
							\vertex (11)[label=above:$v_1$] at (6.2,1.3) {};
							\vertex (12)[label=right:$v_2$] at (7.2, 0.8) {};
							\vertex (13)[label=right:$v_3$] at (7.5, -0.5) {};
							
							\draw(11) to (14);
							\draw(11) to (15);
							\draw(11) to (16);
							\draw(11) to (17);
							\draw(11) to (18);
							\draw(11) to (19);
							
							\draw(12) to (13);
							\draw(12) to (16);
							\draw(12) to (17);
							\draw(12) to (18);
							\draw(12) to (19);
							
							\draw(13) to (16);
							\draw(13) to (17);
							\draw(13) to (18);
							\draw(13) to (19);
							
							\draw(14) to (15);
							\draw(14) to (17);
							\draw(14) to (18);
							\draw(14) to (19);
							
							\draw(15) to (18);
							\draw(15) to (19);
							
							\draw(16) to (19);

							\end{tikzpicture}
							\caption{Graphs $G_9$ (left) and $\overline{G_9}$ (right) in Example \ref{examplen}.}
							\label{pptn}
						\end{center}
					\end{figure}
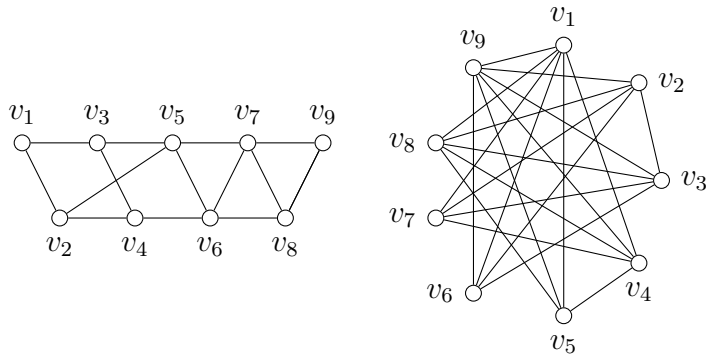
					
				}
				
			\end{ex}
			\begin{conj}\label{ng conjecture} For all graphs $G$ on $n$ vertices, $\ppt(G)+\ppt(\overline{G})\leq n$.
			\end{conj}
			
			We  now show that the conjecture is true for graphs satisfying certain conditions.

\begin{prop}\label{NGtrees} Let $G\neq P_4$ be a connected graph on $n$ vertices that has a leaf. Then $\ppt(G)+\ppt(\overline{G})\leq n-1$ and this bound is tight. For $G=P_4$, $\ppt(G)+\ppt(\overline{G})=n=4$.
\end{prop}

			\begin{proof} The claim holds when $n\leq 2$, so let $n\geq 3$. We first show that $\ppt(\overline{G})\leq 2$. Let $uv\in E(G)$ such that $v$ is a leaf. If $\deg(u)= n-1$, then $\{v,u\}$ is an efficient power dominating set for $\overline{G}$ and $\ppt(\overline{G})=1$. If $\deg(u)\neq n-1$, then $\{v\}$ is an efficient power dominating set for $\overline{G},$ and $\ppt(\overline{G})=2.$  \\\indent Suppose first that $\Delta(G)\geq 3$. By Lemma \ref{HHHHminpowerdomwithdegthree}, $G$ has a minimum power dominating set $S$ such that each vertex in $S$ has degree at least 3. Then $|N[S]|\geq 4$, $\ppt(G)\leq n-3$, and $\ppt(G)+\ppt(\overline{G})\leq n-1$. If $\Delta(G)=2$, then $G$ is a path. By Proposition \ref{proppath}, $\ppt(P_n)=\left \lfloor \frac{n}{2} \right\rfloor$, so $\ppt(P_n)\leq n-3$ for all $n\geq 6.$ For $P_3,P_4,P_5$, we have by inspection that $\ppt(P_3)+\ppt(\overline{P_3})=2,\ppt(P_4)+\ppt(\overline{P_4})=4,$ and $\ppt(P_5)+\ppt(\overline{P_5})=4$. Thus,  $\ppt(G)+\ppt(\overline{G})\leq n-1$ for all graphs $G\neq P_4$ containing a leaf. The bound is tight for $G=\text{sp}(1,1,t)$ ($t\geq 2$) since $\ppt(G)=t=|G|-3$ by Proposition \ref{propspider} and $\ppt(\overline{G})=2$.
\end{proof}
			
			The {\em girth} of a graph is defined to be the length of the shortest cycle contained in the graph. If the graph is acyclic, the girth is defined to be infinity. We now show that Conjecture \ref{ng conjecture} is true for all graph with girth at least 5. 
			
			\begin{thm}\label{ngboundgirth5}
				Let $G$ be a graph on $n\geq 5$ vertices that has girth at least 5. Then $\ppt(\overline{G})\leq 3.$
			\end{thm}
			
			\begin{proof} Let $S'$ be an efficient power dominating set for $\overline{G}$. We will show that  $|N[S']|\geq n-2$. Then it follows from Observation \ref{obs} that $\ppt(\overline{G})\leq 3$. \\\indent Assume that $|N[S']|\leq n-3$, so that $V\setminus N[S'] \geq 3$. Let $u$ be in $V\setminus N[S']$ such that $u$ is forced by some $v\in N[S]\setminus S$ in step 2. Recall that in order for $v$ to force $u$ in step 2, $u$ must be the only neighbor of $v$ in $V\setminus N[S']$.  Let $x$ and $w$ be two vertices in $V\setminus N[S']$ such that $x\neq u$ and $w\neq u$. We first show that $x$ and $w$ must be adjacent. Since $G$ has no 3 cycles, then for any three vertices in $V(\overline{G})$, two of them must be adjacent. Choose $s\in S'$ such that $v\in N(s)$ (this $s$ is guaranteed since $v\in N[S']\setminus S'$). Note that $x,w\notin N(s)$, so $x$ and $w$ must be adjacent. Then the graph induced by $\{x,w,s, v\}$ is $K_2\cup K_2=\overline{C_4}$. This contradicts the hypothesis that the girth of $G$ is at least 5. So $|N[S']|\geq n-2$ and $\ppt(\overline{G})\leq 3$.
				
			\end{proof}

			\begin{cor}
			Let $G$ be a graph on $n$ vertices with girth at least 5. Then $\ppt(G)+\ppt(\overline{G})\leq n$. 
			\end{cor}
			
			\begin{proof}
				It follows from inspection that the claim holds for $n\leq 4$. Assume $n\geq 5$. By Theorem  \ref{ngboundgirth5}, $\ppt(\overline{G})\leq 3$. Suppose first that $\Delta(G)\geq 3$, and let $G_1$ be the connected component of $G$ that has a vertex of degree at least 3. Then there exists a minimum power domination set $S_1$ of $G_1$ such that each vertex in $S_1$ has degree at least 3 (Lemma \ref{HHHHminpowerdomwithdegthree}). Therefore, $|N[S_1]|\geq 4$, and for any minimum power dominating set $S$ of $G$ with $S_1\subseteq S$, $|N[S]|\geq 4$, so $\ppt(G)\leq \ppt(G,S)\leq n-3$ (Observation \ref{obs}).  This gives that $\ppt(G)+\ppt(\overline{G})\leq n.$ \\\indent If $\Delta(G)\leq 2$, then $G$ is the union of paths and cycles, and the power propagation time of $G$ is equal to the power propagation time of the path or cycle with the largest number of vertices. This component has at most $n$ vertices, so its power propagation time of this component is at most $\left \lfloor \frac{n}{2} \right \rfloor$ (Propositions \ref{proppath} and \ref{propcycle}).	It follows that $\ppt(G)\leq \left \lfloor \frac{n}{2} \right \rfloor$ , and since $n\geq 5$, $\ppt(G)+\ppt(\overline{G})\leq n$.	\end{proof}
			
			\begin{lem}\label{sizedecrease}{\rm \cite{Chang}} Let $G$ be a connected graph such that $\Delta(G)\geq 3$. Then there exists a minimum power dominating set $S$ such that each $s\in S$ has at least two neighbors which are not in $N[S\setminus \{v\}].$
				
			\end{lem}

			\begin{prop}\label{delta3delta3} Let $G$ and $\overline{G}$ be connected graphs on $n$ vertices such that $\Delta(G)\geq 3$ and $\Delta(\overline{G})\geq 3.$ Then $\ppt(G)+\ppt\left (\overline{G}\right)\leq n-(\gamma_P(G)+\gamma_p(\overline{G}))+4$.
			\end{prop}

			\noi {\em Proof.} By Lemma \ref{sizedecrease} and the assumption that $\Delta(G)\geq 3$, there is a minimum power dominating set  $S$ of $G$ such that each $s\in S$ has at least one neighbor not in $N[S\setminus\{s\}]$. We first show that $\ZFN(G)\leq |N[S]|-\gamma_p(G)$. Recall that $N[S]$ is a zero forcing set of $G$. For each $s\in S$, choose a $v_s\in N(s)$ such that $v_s\notin N[S\setminus\{s\}].$ Then $N[S]\setminus\{v_1,v_2,\ldots,v_{|S|}\}$ is also a zero forcing set since $s$ will force $v_s$ in step one. So,  $\ZFN(G)\leq |N[S]|-\gamma_p(G)$.\\\indent By the same argument, we have a minimum power dominating set $S'$ of $G$ such that $\ZFN(\overline{G})\leq |N[S']|-\gamma_p(\overline{G})$. Using the bounds $\ppt(G,S)-1\leq n-|N[S]|$ and $\ppt(\overline{G},S')-1\leq n-|N[S']|$ (from inequality (\ref{pptbound2})), and $n-2\leq \ZFN(G)+\ZFN(\overline{G})$ from \cite{PSDZF}, it follows that

			\begin{eqnarray*}
				\ppt(G)+\ppt(\overline{G})&\leq & \ppt(G,S)+\ppt(\overline{G},S')\\
				&\leq & 2n+2 -(|N[S]|+|N[S']|)\\
				&\leq & 2n+2-(\ZFN(G)+\ZFN(\overline{G}))-(\gamma_P(G)+\gamma_P(\overline{G}))\\
				&\leq & 2n+2-(n-2)-(\gamma_P(G)+\gamma_P(\overline{G}))\\
				&= & n-(\gamma_P(G)+\gamma_P(\overline{G}))+4. \qed
			\end{eqnarray*} 
			
			\begin{cor}
				Let $G$ and $\overline{G}$ be connected graphs on $n$ vertices with $\gamma_{P}(G)+\gamma_{P}(\overline{G})\geq 4.$ Then $ \ppt(G)+\ppt\left (\overline{G}\right)\leq n.$
			\end{cor}
		
\begin{proof}
We first show that $\Delta(G)\geq 3$ and $\Delta(\overline{G})\geq 3$. If $\Delta(G)\leq 2$ then $G$ is a cycle or a path. By the assumption that $G$ and $\overline{G}$ are connected, $G\notin \{P_2, P_3, C_3, C_4\}$. For $n\geq 4$, $\gamma_P(P_n)=\gamma_P(\overline{P_n})
=1$ and for $n\geq 5, \gamma_P(C_n)=\gamma_P(\overline{C_n})=1.$ It follows from the assumption that $\gamma_P(G)+\gamma_P(\overline{G})\geq 4$ that neither $G$ or $\overline{G}$ is a path or cycle. Thus, $\Delta(G)\geq 3$ and $\Delta(\overline{G})\geq 3$. By Proposition \ref{delta3delta3}, \[\ppt(G)+\ppt(\overline{G})\leq n-(\gamma_P(G)+\gamma_P(\overline{G}))+4\leq n.\] 

\end{proof}
			\section{Effects of edge subdivision and edge contraction on power propagation time}\label{operations}

			Let $G_e$ be a graph obtained from $G=(V,E)$ by subdividing the edge $e\in E$ and let $G/e$ denote the graph resulting from $G$ by contracting the edge $e$. It is shown in both \cite{PD2015} and \cite{arxivpaper} that $\gamma_P(G)-1 \leq\gamma_P(G/e)\leq \gamma_P(G)+1$ and in \cite{PD2015} that $\gamma_P(G) \leq\gamma_P(G_e)\leq \gamma_P(G)+1$. We show that the power propagation time may increase or decrease by any amount when subdividing or contracting an edge.

			\begin{prop}\label{subdecrease} For any $t\geq 0$, there exists a graph $G=(V,E)$ and edge $e\in E$ such that $\ppt(G_e)\leq \ppt(G)-t.$
				
			\end{prop}

			\begin{proof} Construct the graph $G$ in the following way: Starting with the path $P_{\ell}=(v_1,v_2,\ldots, v_{\ell}),$ $(\ell\geq 7)$, add three leaves to vertex $v_1$ and add three leaves to vertex $v_{\ell}$. Add one leaf to vertex $v_{\ell-1}$ and add one leaf to vertex $v_{\ell-2}$. (See Figure \ref{subdecreasefig}.) Then $\{v_1,v_{\ell}\}$ is the unique efficient power dominating set of $G$ and $\ppt(G)=\ell-2$. For $e=\{v_{l-2}v_{l-1}\}$, we consider the graph $G_e$. Note that $\gamma_p(G_e)=3$ because $v_1, v_{\ell}\in S$ for any  minimum power dominating set $S$ and $\{v_1, v_{\ell}\}$ is not a power dominating set. For $S=\{v_1, v_{l-2}, v_{\ell}\}, \ppt(G_e,S)=\left \lceil \frac{\ell-4}{2} \right \rceil$. By choosing $\ell\geq 2t+1$, $\ppt(G_e)\leq \ppt(G)-t.$
				
			\end{proof}
			
			\begin{cor}\label{contractincrease} For any $t\geq 0$, there exists a graph $H=(V,E)$ and edge $e\in E$ such that $\ppt(H/e)\geq \ppt(H)+t.$
			\end{cor}

			\begin{proof} From Proposition \ref{subdecrease}, there exist graphs $G$ and $G_e$ such that $\ppt(G_e)\leq \ppt(G)-t.$ Let $H=G_e$ and $H/e=G$. Then $\ppt(H/e)\geq \ppt(H)+t.$
				
			\end{proof}
			
			\begin{figure}[h!]
				\begin{center}
					\begin{tikzpicture}[scale=1.0]
					
					
					\vertex [label=below:$v_1$](1) at (0,0) {};
					\vertex [label=below:$v_2$](2) at (1,0){};
					\vertex [label=right:$\ldots$](3) at (2,0) {};
					\vertex (4) at (3,0){};
					\vertex [label=below:$v_{\ell-2}$](5) at (4,0) {};
					\vertex [label=below:$v_{\ell-1}$](6) at (5,0){};
					\vertex [label=below:$v_{\ell}$](7) at (6,0){};
					\vertex (8) at (5,1){};
					\vertex (9) at (4,1){};
					\vertex (10) at (-1,0.5) {};
					\vertex (11) at (-1,-0.5) {};
					
\vertex (111) at (-1,0) {};
					\vertex (12) at (7,0.5) {};

\vertex (121) at (7,0) {};					
					\vertex (13) at (7,-0.5) {};
					
					\draw (1) to (2);
					\draw(2) to (3);
					\draw (4) to (5);
					\draw(5) to (6);
					\draw (6) to (7);
					\draw(8) to (6);
					\draw (9) to (5);
					\draw(1) to (11);
					
\draw(1) to (111);
					\draw(1) to (10);
					\draw(7) to (12);
					\draw(7) to (121);
					\draw(7) to (13);
					
					
					\vertex [label=below:$v_1$](21) at (0,-3) {};
					\vertex [label=below:$v_2$](22) at (1,-3){};
					\vertex [label=right:$\ldots$](23) at (2,-3) {};
					\vertex (24) at (3,-3){};
					\vertex [label=below:$v_{\ell-2}$](25) at (4,-3) {};
					\vertex (214) at (4.5,-3) {};
					\vertex [label=below:$v_{\ell-1}$](26) at (5,-3){};
					\vertex [label=below:$v_{\ell}$](27) at (6,-3){};
					\vertex (28) at (5,-2){};
					\vertex (29) at (4,-2){};
					\vertex (210) at (-1,-2.5) {};
					
\vertex (2100) at (-1,-3) {};
					\vertex (211) at (-1,-3.5) {};
					\vertex (212) at (7,-2.5) {};
							\vertex (213) at (7,-3.5) {};
	
\vertex (215) at (7,-3) {};				
					\draw(25) to (214);
					\draw(214) to (26);
					\draw (21) to (22);
					\draw(22) to (23);
					\draw (24) to (25);
					\draw (26) to (27);
					\draw(28) to (26);
					\draw (29) to (25);
					\draw(21) to (211);
					\draw(21) to (210);
				
\draw(21) to (2100);			

\draw(27) to (212);
					\draw(27) to (213);
					\draw(27) to (215);

					\end{tikzpicture}
					\caption{Graphs $G$ and $G_e$ in Proposition \ref{subdecrease}.}
					\label{subdecreasefig}
				\end{center}
			\end{figure}
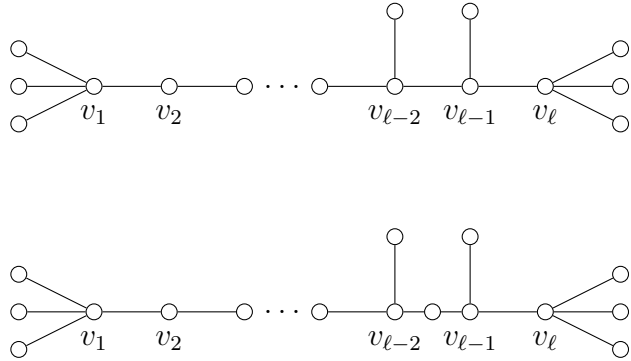
			
			Similarly, subdividing an edge can cause the power propagation time to increase by any amount, as demonstrated by the following proposition.
			
			\begin{prop}\label{subincrease} For any $t\geq 0$, there exists a graph $G=(V,E)$ and edge $e\in E$ such that $\ppt(G_e)\geq \ppt(G)+t.$
				
			\end{prop}

			\begin{proof}
				
				Let $G$ be a graph on $n\geq 8$ vertices constructed from the cycle $(v_1,v_2,\ldots, v_{n-4})$ by adding the edges $\{v_1,v_{n-3}\}, \{v_1,v_{n-2}\}, \{v_1,v_{n-1}\}, \{v_2,v_{n-1}\}, \text{ and } \{v_n,v_{n-1}\}$. Let $e=\{v_2,v_{n-1}\}$, and consider $G_e$. The set $\{v_1\}$ is the unique minimum power dominating set of $G$ and $\ppt(G)=\left \lfloor \frac{n-4}{2}\right \rfloor.$ The set $\{v_1\}$ is also the unique minimum power dominating set of $G_e$ and $\ppt(G_e)=n-4$. So, by choosing $n \geq 2t+4,$ $\ppt(G_e)\geq \ppt(G)+t.$ \end{proof}

			\begin{figure}[h!]
				\begin{center}
					\begin{tikzpicture}[scale=1.0]
					
					
					\vertex [label=left:$v_{n-2}$](30) at (7.5,-2) {};
					\vertex [label=left:$v_{n-3}$](40) at (6.6,-2.6) {};
					\vertex  [label=above:$v_{n-1}$] (31) at (8.3,-2) {};
					\vertex  [label=left:$v_{1}$](32) at (7.5,-3) {};
					\vertex  [label=right:$v_{2}$](33) at (8.3,-3) {};
					\vertex [label=left:$v_{n-4}$] (34) at (7,-3.5) {};
					\vertex [label=left:$v_{n-5}$](35) at (7,-4.3) {};
					\vertex  [label=right:$\ldots$](36) at (7.5,-4.6) {};
					\vertex  (37) at (8.4,-4.6) {};
					\vertex  [label=right:$v_{4}$](38) at (8.9,-4.1) {};
					\vertex  [label=right:$v_{3}$](39) at (8.9,-3.5) {};
					\vertex  [label=right:$v_{n}$](41) at (9,-2) {};
					
					\draw(30) to (32);
					\draw(40) to (32);
					\draw(31) to (32);
					\draw(31) to (41);
					\draw(31) to (33);
					\draw(32) to (34);
					\draw(34) to (35);
					\draw(35) to (36);
					\draw(32) to (33);
					\draw(37) to (38);
					\draw(38) to (39);
					\draw (39) to (33);

					
					\vertex [label=left:$v_{n-2}$](130) at (12.5,-2) {};
					\vertex [label=left:$v_{n-3}$](140) at (11.6,-2.4) {};
					\vertex  [label=above:$v_{n-1}$] (131) at (13.3,-2) {};
					\vertex  [label=left:$v_{1}$](132) at (12.5,-3) {};
					\vertex  [label=right:$v_{2}$](133) at (13.3,-3) {};
					\vertex [label=left:$v_{n-4}$] (134) at (12,-3.5) {};
					\vertex [label=left:$v_{n-5}$](135) at (12,-4.3) {};
					\vertex  [label=right:$\ldots$](136) at (12.5,-4.6) {};
					\vertex  (137) at (13.4,-4.6) {};
					\vertex  [label=right:$v_{4}$](138) at (13.9,-4.1) {};
					\vertex  [label=right:$v_{3}$](139) at (13.9,-3.5) {};
					\vertex  [label=right:$v_{n}$](141) at (14,-2) {};
					\vertex (143) at (13.3,-2.5) {};

					\draw(130) to (132);
					\draw(140) to (132);
					\draw(131) to (132);
					\draw(131) to (141);
					\draw(143) to (133);
					\draw(143) to (131);
					\draw(132) to (134);
					\draw(134) to (135);
					\draw(135) to (136);
					\draw(132) to (133);
					\draw(137) to (138);
					\draw(138) to (139);
					\draw (139) to (133);

					\end{tikzpicture}
					\caption{Graphs $G$  and $G_e$ in Proposition \ref{subincrease}.}
					\label{subincreasefig}
				\end{center}
			\end{figure}
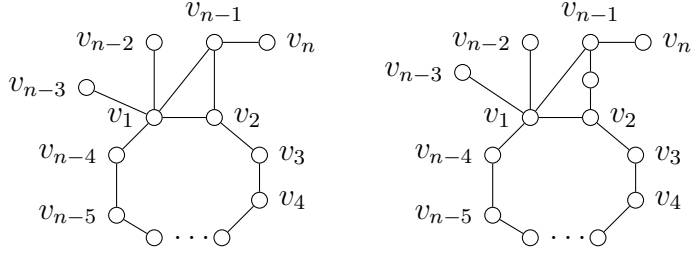

			\begin{cor}\label{contractdecrease} For any $t\geq 0$, there exists a graph $H=(V,E)$ and edge $e\in E$ such that $\ppt(H/e)\leq \ppt(H)-t.$
				
			\end{cor}
			
			\begin{proof}
				
				From Proposition \ref{subincrease}, there exist graphs $G$ and $G_e$ such that $\ppt(G_e)\geq \ppt(G)+t.$ Let $H=G$ and $H/e=G_e$. Then $\ppt(H/e)\leq \ppt(H)-t.$
				
			\end{proof}

      	\section{$k$-power propagation}\label{kprop}

		The authors of \cite{Chang} introduced the following generalization of power domination, known as {\em $k-$power domination.} Let $k\geq 1$. For a set $S\subseteq V(G)$, define the following sets:
			
			\begin{enumerate}
				\item[1.] $S^{[0]}=S, S^{[1]}=N[S].$
				\item[2.]For $t\geq 1$, $S^{[t+1]}=S^{[t]}\cup \{w\in V(G)| \hspace{1mm} \exists \hspace{1mm} v\in S^{[t]}, w\in N(v) \setminus S^{[t]}\text{ and } |N(v) \setminus S^{[t]}|\leq k\}$.
			\end{enumerate}

			\noi(For our purposes and convenience, we have defined $S^{[0]}=S$. This is not done in \cite{Chang}.) A set $S$ is said to be a {\em $k-$power dominating set} if there exists an $l$ such that $S^{[l]}=V(G).$ (Note that when $k=1$ the set is a power dominating set.) The {\em $k-$power domination number} of $G$, denoted $\gamma_{P,k}(G)$, is defined to be the minimum cardinality over all $k-$power dominating sets of $G$, and $\gamma_{P,k}(G)\leq \gamma_P(G)\leq \gamma(G)$ for all $k\geq 1$ \cite{Chang}. \\\indent  We define the $k-$power propagation time as follows:

\begin{defn}{\rm Let $S$ be a $k-$power dominating set. The {\em $k-$power propagation time of $G$ with $S$}, denoted $\ppt_k(G,S)$ is the smallest $\ell$ such that $S^{[\ell]}=V(G)$. The {\em $k-$power propagation time of $G$}, denoted $\ppt_k(G)$ is given by \[\ppt_k(G)=\min\{\ppt_k(G,S)|S \text{ is a minimum $k-$power dominating set} \}.\]}
\end{defn} 

A minimum $k-$power dominating  set $S$ of a graph $G$ is {\em efficient} if $\ppt_k(G,S)=\ppt_k(G)$. \\\indent  In this section, we study the $k-$power propagation time of a graph by characterizing graphs with extreme high and extreme low $k-$power propagation times. Note that by letting $k=1$, we obtain characterizations of graphs with extreme high and extreme low power propagation times. \\\indent The next observation and next two propositions are generalizations of Observation \ref{obs} and Propositions \ref{proppath} and \ref{propcycle}, and the same arguments hold. 
			
				\begin{obs}\label{obs2}{\rm
						
						Let $G$ be a graph on $n$ vertices and $S$ a $k-$power dominating set of $G$. Then,
						
						\begin{equation}
						\ppt_k(G,S)\leq n-|S| \\ \label{kpptbound1}
						\end{equation} and
						
						\begin{equation}
						\ppt_k(G,S)-1\leq n-|N[S]| \\ \label{kpptbound2}
						\end{equation}
					}\end{obs}
				
					\begin{prop}\label{kproppath}Let $P_n$ be the path on $n$ vertices. Then $\ppt_k(P_n)=\left \lfloor \frac{n}{2} \right \rfloor$.
					\end{prop}

					\begin{prop}\label{kpropcycle}Let $C_n$ be the cycle on $n$ vertices. Then $\ppt_k(C_n)=\left \lfloor \frac{n}{2} \right \rfloor$.						
					\end{prop}

			\begin{rem}\label{remarkk2}{\rm It is a well known fact that for  a connected graph $G$ of order at least 3, there exists an efficient $k-$power dominating set of $G$ in which every vertex has degree at least 2: For if $v$ is a leaf of an efficient $k-$power dominating set $S$ and $vw\in E(G)$, then $w$ is not a leaf since $G$ is connected and $G\neq K_2$. So, $S'=(S \setminus \{v\}) \cup \{w\}$ is a minimum $k-$power dominating set, and $\ppt_k(G, S')\leq \ppt_k(G,S)$. Repeating this process for each leaf in $S$, we obtain an efficient $k-$power dominating set of $G$ with no leaves.
					
				}
				
			\end{rem}

			\begin{lem}\label{minpowerdomwithdegthree} {\rm \cite{Chang}} Let $k\geq 1$ and let $G$ be a connected graph with $\Delta(G)\geq k+2$. Then there exists a minimum $k-$power dominating set $S$ of $G$ such that $\deg(s)\geq k+2$ for each $s\in S.$
				
			\end{lem}

			Note that $\Delta(G)\geq k+2$ does not guarantee that there exists an efficient $k-$power dominating set $S$ such that $\deg(s)\geq k+2$ for each $s\in S.$  This is demonstrated in the following example with $k=1$.

			\begin{ex}{\rm
					
					Let $G$ be the graph on $n+2$ vertices ($n\geq 5$) obtained from a path $(v_1,v_2,\ldots, v_n)$ by adding a leaf to $v_2$ and adding a leaf to $v_3$. Then $S=\{v_2,v_3\}$ is the unique power dominating set such that $\deg(s)\geq 3$ for each $s\in S,$ but for $S'=\{v_2,v_4\}$, $n-4=\ppt(G,S')< \ppt(G,S)=n-3$.
					
				}
				
			\end{ex}
			
			Throughout the rest of this paper, we also use the following generalization of Lemma \ref{minpowerdomwithdegthree}:
			
			\begin{lem}\label{deltat}{\rm
					
					For any $3\leq t \leq k+2 $, if $G$ is connected with $\Delta(G) \geq t,$ then there exists a minimum $k-$power dominating set $S$ such that every vertex in $S$ has degree at least $t$.
					
				}
				
			\end{lem}
			
			\begin{proof}
				Let $3\leq t\leq k+2$ and let $S$ be a minimum $k-$power dominating set of $G$. Suppose $s\in S$ and $\deg(s)<t$. Since $G$ is connected, we may choose $v\in V(G)$ such that $\deg(v)\geq t$  and $\deg(u)<t$ for all interior vertices $u$ on the shortest path from $s$ to $v$. Then $(S\setminus\{s\})\cup\{v\}$ is also a minimum $k$-power dominating set. Continuing this process for all vertices in $S$ with degree less than $t$, we construct a minimum  $k$-power dominating set of $G$ such that every vertex has degree at least $t$.
			\end{proof} 
		\subsection{Low $k$-power propagation time}\label{low}
		We first consider graphs with low $k-$propagation time. If $G$ is a graph with $k$-propagation time 1, then any efficient $k$-power dominating set of $G$ is also a dominating set, so $\gamma(G)\leq \gamma_{P,k}(G)$. Since it is always true that $\gamma_{P,k}(G)\leq \gamma(G),$ it follows that $\gamma_{P,k}(G)=\gamma(G).$ In this section, we study graphs with $k-$power propagation time equal to 1. \\\indent For $k\geq 1$, a vertex $v$ in $V(G)$ is called a {\em k-strong support vertex} if $v$ is adjacent to $k+1$ or more leaves. A $1-$strong support vertex is also known as a {\em strong support vertex} and was originally defined in \cite{HHHH}.

		\begin{rem}\label{powerequalsdom}{\rm
				
				Note that every $k$-strong support vertex of a graph $G$ is in every minimum dominating set of $G$. Also, if $S$ is a $k-$power dominating set of $G$ and $v$ is a $k$-strong support vertex of $G$ then either $v$ is in $S$ or all but $k$ of the leaves adjacent to $v$ are in $S$. So $\gamma_{P,k}(G)$ is at least the number of $k$-strong support vertices in $G$. Since $\gamma_{P,k}(G)\leq \gamma(G)$, it follows that if $S$ is a dominating set of $G$ such that every vertex in $S$ is a $k$-strong support vertex, then $S$ is the unique minimum dominating set of $G$, $\gamma_{P,k}(G)=\gamma(G)$, and $\ppt_k(G)=1$.
				
			}
			
		\end{rem}

		\indent For a minimum $k-$power dominating set $S$ and a vertex $v$ in $S$, the {\em private neighborhood} of $v$ with respect to $S$, denoted $pn[v,S]$, is the set $N[v]\setminus (N[S\setminus\{v\}])$. Every vertex of $pn[v,S]$ is called a {\em private neighbor} of $v$ with respect to $S$, and $A_v$ denotes the set $V\setminus(S\cup pn[v,S])$ \cite{HHHH}. 
		
		
The next theorem and proof is a generalization of  Theorem 9 given in \cite{HHHH}.

\begin{thm}\label{ppt1girth5} For $k\geq 1$, let $G$ be a connected graph on at least k+2 vertices that does not contain $C_3$ or $K_{2,k+1}$ as an induced subgraph. Then $\ppt_k(G)=1$ if and only if $G$ has a minimum dominating set $S$ such that every vertex in $S$ is a $k$-strong support vertex. Furthermore, $S$ is the unique minimum dominating set of $G$.
\end{thm}

		\begin{proof}
			
			If $G$ has a dominating set $S$ such that each vertex in $S$ is a $k-$strong support vertex, then by Remark \ref{powerequalsdom}, $\gamma_{P,k}(G)=\gamma(G)$ and $\ppt_k(G)=1$. \\\indent Conversely, let $\ppt_k(G)=1$ (i.e $\gamma_{P,k}(G)=\gamma(G)$). To obtain a contradiction, suppose $S$ is a minimum dominating set of $G$ such that there exists a vertex $v\in S$ that is not a $k-$strong support vertex. If $pn[v,S]=\emptyset,$ then $S\setminus\{v\}$ is a smaller dominating set. Suppose that $pn[v,S]=\{v\}$. Then $S\setminus\{v\}$ dominates $V\setminus\{v\}$, and since $G$ is connected, $v$ will be forced in step 1. So $S\setminus\{v\}$ is a smaller $k$-power dominating set. Thus, $pn[v,S]$ contains at least one vertex that is not $v$. \\\indent Suppose first that $pn[v,S]$ contains a vertex $w\neq v$ that is not a leaf. We show again that $S\setminus\{v\}$ is a smaller $k$-power dominating set. Since $w$ is not a leaf, it is adjacent to a vertex in $A_v$: To see this, note that $w$ has no neighbor in $pn[v,S]$ (except for $v$ if $v$ is in $pn[v,S]$) since every other vertex in $pn[v,S]$ is also adjacent to $v$ and $G$ contains no 3 cycles. Furthermore, by the definition  of $pn[v,S]$, $w$ has no neighbor in $S\setminus \{v\}$. Since $w$ is not a leaf, then $w$ is adjacent to some vertex $w_u$ in $A_v$. To see that $S\setminus\{v\}$ is a smaller $k$-power dominating set,  first note that  $w_u$ is not adjacent to $v$ (since $(v,w_u,w)$ could give a 3 cycle) and  $|N(w_u)\cap (pn[v,S]\setminus\{v\})|\leq k$ (since $G$ is $K_{2,k+1}$-free and the vertices of $N(w_u)\cap (pn[v,S]\setminus\{v\})$ form the induced graph $K_{2,t}$ where $t=|N(w_u)\cap (pn[v,S]\setminus\{v\})|$).  It follows that $S\setminus\{v\}$ is a $k-$power  dominating set of $G$ since  $S\setminus\{v\}$ dominates $A_v$ in step 1, each $w$ in $pn[v,S]\setminus\{v\}$ that is not a leaf is forced by a neighbor $w_u$ from $A_v$ step 2, if necessary any such $w$ can force $v$ in step 3, and since $v$ is adjacent to at most $k$ leaves, then $v$ will force these leaves (if any) in step 4. So each vertex in $pn[v,S]$ that is not $v$ must be a leaf.\\\indent Suppose  vertices $w_1,...,w_t$ are leaves in $pn[v,S]$,  where $1\leq t\leq k$ since $v$ is not a $k-$strong support vertex. Since $G$ is connected and each $w_i$ is only adjacent to $v$, $v$ must have a neighbor in $S\setminus \{v\}$ or in $A_v$. In either case, we show that $S\setminus \{v\}$ is a smaller $k-$power dominating set. If $v$ has a neighbor in $S\setminus \{v\}$, then $S\setminus \{v\}$ dominates $A_v\cup \{v\}$ in step 1, and $v$ will force $\{w_1,...,w_t\}$ in step 2. If $v$ has a neighbor in $A_v$ (and no neighbor in $S\setminus \{v\}$), then $S\setminus \{v\}$ dominates $A_v$ in step 1, $v$ is forced by a neighbor from $A_v$ in step 2, and $v$ forces $w_1,...,w_t$ in step 3. This completes the proof of the first statement in the theorem. Note that we have shown that if $\ppt_k(G)=1$, then every minimum dominating set of $G$ contains only $k-$strong support vertices. Thus, if $S$ is a minimum dominating set such that each vertex in $S$ is a $k-$strong support vertex, then $S$ is the unique minimum dominating set of $G$.
\end{proof}

		\subsection{High $k$-power propagation time}\label{high}
		
		Here we consider graphs with high $k-$power propagation times. First we characterize all graphs  on $n$ vertices with $\ppt_k(G)=n-1$ or $\ppt_k(G)=n-2$.

		\begin{thm}\label{orderminus1} For a graph $G$ on $n$ vertices and $k\geq 1$, $\ppt_k(G)=n-1$ if and only if $G=K_1$ or $G=K_2.$
			
		\end{thm}

		\begin{proof} Let $S$ be an efficient $k$-power dominating set of $G$. Since $\ppt_k(G)=n-1$, then $S=\{s\}$ for some $s\in V(G)$, and $G$ is connected. Note that at most 1 vertex may be forced at each step, including the domination step, so $\deg(s)\leq 1$. By Remark \ref{remarkk2}, $n\leq 2$, so $G=K_1$ or $G=K_2$.
		\end{proof}

		\begin{thm}\label{orderminus2} Let $k\geq 1$ and let $G$ be a graph on $n$ vertices with $\ppt_k(G)=n-2$. Then $G\in \{K_1\cup K_1, K_1\cup K_2, P_3, P_4, C_3, C_4\}$.
			
		\end{thm}
		
		\begin{proof}

			Since $\ppt_k(G)=n-2$, then for any minimum $k-$power dominating set $S$, $|S|\leq 2$ and $|N[S]|\leq 3$. Suppose $\Delta(G)\geq 3$ and let $G_1$ be a connected component of $G$ with $\Delta(G_1)\geq 3$. By Lemma \ref{deltat}, there exists a minimum $k-$power dominating set $S_1$ of $G_1$ such that each $s\in S_1$ has degree at least 3. Then for any minimum $k-$power dominating set $S$ of $G$ such that $S_1\subseteq S$, we have that $|N[S]|\geq 4$, contradicting $|N[S]|\leq 3$. So $\Delta(G)\leq 2$ and $G$ is the union of cycles and paths. Since $|S|\leq 2$, then $G$ has at most 2 components. If $G$ has exactly one component, $G$ is a path or a cycle, and it follows from Propositions \ref{kproppath} and \ref{kpropcycle} that $G\in \{P_3, P_4, C_3, C_4\}$. Suppose $G$ has 2 components. Since $|N[S]|\leq 3$, one component is $K_1$, and by Remark \ref{remarkk2} (or Theorem \ref{orderminus1}), the other component is $K_1$ or $K_2$.
		\end{proof}

		Next we consider graphs $G$ whose $k-$power propagation time is $n-3$. The case with $k=1$ behaves differently than the cases with $k\geq 2$, so we first consider the latter.\\\indent We use $\mathfrak{G}$ to denote the family of connected graphs $G$  on 5 vertices with $\Delta(G)=3$ (see Figure \ref{graphs5delta3}).
		
		\begin{figure}[h!]
			\begin{center}
				\begin{tikzpicture}[scale=1.0]
				\vertex (1) at (0,0) {};
				\vertex (2) at (-.5,-.5) {};
				\vertex (3) at (.5,-.5) {};
				\vertex (4) at (-.5,-1) {};
				\vertex (5) at (.5,-1) {};
				
				\draw(1) to (2);
				\draw(1) to (3);
				\draw(2) to (3);
				\draw(3) to (5);
				\draw(2) to (4);
				\draw(5) to (4);

				\vertex (6) at (2,0) {};
				\vertex (7) at (2.5,0) {};
				\vertex (8) at (2,-.5) {};
				\vertex (9) at (2.5,-.5) {};
				\vertex (10) at (2,-1) {};
				
				\draw(6) to (7);
				\draw(8) to (9);
				\draw(8) to (10);
				\draw(6) to (8);
				\draw(7) to (9);
				
				\vertex (11) at (4,0) {};
				\vertex (12) at (3.5,-.5) {};
				\vertex (13) at (4.5,-.5) {};
				\vertex (14) at (3.5,-1) {};
				\vertex (15) at (4.5,-1) {};
				
				\draw(11) to (12);
				\draw(11) to (13);
				\draw(12) to (13);
				\draw(13) to (15);
				\draw(12) to (14);
				
				\vertex (16) at (6,0) {};
				\vertex (17) at (6,-.5) {};
				\vertex (18) at (5.5,-.5) {};
				\vertex (19) at (6.5,-.5) {};
				\vertex (20) at (6,-1) {};
				
				\draw(16) to (18);
				\draw(16) to (19);
				\draw(17) to (19);
				\draw(17) to (18);
				\draw(20) to (18);
				\draw(20) to (19);

				\vertex (21) at (6,-2) {};
				\vertex (22) at (6,-2.5) {};
				\vertex (23) at (5.5,-2.5) {};
				\vertex (24) at (6.5,-2.5) {};
				\vertex (25) at (6,-3) {};
				
				\draw(21) to (23);
				\draw(21) to (22);
				\draw(21) to (24);
				\draw(22) to (24);
				\draw(22) to (23);
				\draw(25) to (23);
				\draw(25) to (24);
				
				\vertex (26) at (-.5,-2.5) {};
				\vertex (27) at (0,-2.5) {};
				\vertex (28) at (-.5,-3) {};
				\vertex (29) at (0,-3) {};
				\vertex (30) at (-.5,-2) {};
				
				\draw(26) to (27);
				\draw(27) to (28);
				\draw(28) to (29);
				\draw(29) to (27);
				\draw(28) to (26);
				\draw(26) to (30);
				
				\vertex (31) at (1.5,-1.5) {};
				\vertex (32) at (2.5,-1.5) {};
				\vertex (33) at (2,-2) {};
				\vertex (34) at (2,-2.5) {};
				\vertex (35) at (2,-3) {};

				\draw(31) to (32);
				\draw(31) to (33);
				\draw(32) to (33);
				\draw(33) to (34);
				\draw(34) to (35);
				
				\vertex (36) at (3.5,-2) {};
				\vertex (37) at (4.5,-2) {};
				\vertex (38) at (4,-2) {};
				\vertex (39) at (4,-2.5) {};
				\vertex (40) at (4,-3) {};

				\draw(36) to (38);
				\draw(37) to (38);
				\draw(38) to (39);
				\draw(39) to (40);

				\end{tikzpicture}
				
				\caption{$\mathfrak{G}$: Connected graphs $G$ on 5 vertices with $\Delta(G)=3$.}
				\label{graphs5delta3}
			\end{center}
		\end{figure}
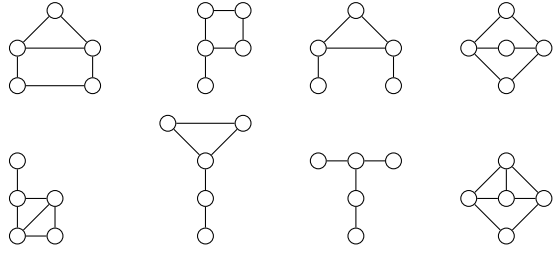
		
		\begin{thm}\label{orderminus2} Let $k\geq 2$ and let $G$ be a graph on $n$ vertices with $\ppt_k(G)=n-3$. Then $G\in \{P_5, P_6, C_5, C_6,K_{1,3}, L(3,1), K_4-e, K_4, K_1\cup P_3, K_1\cup P_4, K_1\cup C_3, K_1\cup C_4,K_2\cup K_2,\overline{K_3}, \overline{K_2} \cup K_2\}\cup \mathfrak{G}$.
			
		\end{thm}

		\begin{proof}Since $\ppt_k(G)=n-3,$ then for any minimum $k-$power dominating set $S$, $|S|\leq 3$ and $|N[S]|\leq 4.$ It follows from Lemma \ref{deltat} that $\Delta(G)\leq 3$. \\\indent If $\Delta(G)\leq 2$, then $G$ is the  union of paths and cycles. Since $|S|\leq 3$ for any minimum $k$-power dominating set $S$, $G$ has at most 3 components. If $G$ is connected, it follows from Propositions \ref{kproppath} and \ref{kpropcycle} that $G\in \{P_5, P_6, C_5, C_6\}.$  \\\indent Suppose $G$ has two connected components, $G_1$ and $G_2$, and first suppose $|G_1|\geq 3$. By applying Remark \ref{remarkk2} to $G_1$, there exists an efficient $k-$power dominating set $S$ of $G$ such that $|N_{G_1}[S]|\geq 3,$ where $N_{G_1}[S]=N[S]\cap V(G_1)$. Since $|N[S]|\leq 4$, we have $G_2=K_1,\ppt_k(G_1)=|G_1|-2,$ and it follows from Theorem \ref{orderminus2} that $G\in\{ K_1\cup P_3, K_1\cup P_4, K_1\cup C_3, K_1\cup C_4\}$. Otherwise, $|G_1|\leq 2$ and $|G_2|\leq 2$, and $G=K_2\cup K_2.$  \\\indent If $G$ has 3 connected components, it follows from $|N[S]|\leq 4$ that $G\in \{\overline{K_3}, \overline{K_2} \cup K_2\}$. \\\indent Suppose $\Delta(G)=3$. Let $S$ be a minimum $k-$power dominating set such that every vertex in $S$ has degree at least 3 ($S$ is guaranteed to exist by Lemma \ref{minpowerdomwithdegthree} ).  Since $|N[S]|\leq 4$, then $S=\{s\}$ and $N[S]=\{s, u_1, u_2, u_3\}$ for some $s, u_1, u_2, u_3\in V(G)$.  If $n=4,$ then $G\in \{K_{1,3}, L(3,1), K_4-e, K_4\}.$ \\\indent For $n>4$, we show that $n=5$: Since $|N[S]|=4$ and $\ppt_k(G)=n-3$, then after the domination step, exactly one force is performed during each step. Without loss of generality, suppose $u_1$ forces $v$ in step 2. \\\\{\em Claim 1:} For $i\in \{1, 2,3\}$, if $w\in N(u_i)$, then $w\in \{s, u_1, u_2, u_3, v\}:$ To see this, recall that $\Delta(G)=3$. So if $u_1$ has a neighbor $w$ not in $\{s, u_2, u_3, v\},$  it has exactly one such neighbor. Then $u_1$ will force $w$ and $v$ in step 2, which contradicts $\ppt_k(G)=n-3$. Similarly, if $u_i$ (for $i=2,3$) has a neighbor $w$ not in $\{s,u_1, u_2, u_3, v\}$, it has at most two such neighbors, so $u_1$ will force $v$ in step 2 and $u_i$ will force $w$ in step 2, contradicting $\ppt_k(G)=n-3.$

			{\em Claim 2:} Vertex $v$ has no neighbor not in $\{u_1, u_2, u_3\}$. To see this, suppose $v$ has a neighbor $w$ not in $\{u_1, u_2, u_3\}$. Since $\Delta(G)=3$ and $v$ is adjacent to $u_1$ by assumption, then $v$ has at most two such neighbors. Then $\{u_1\}$ is a minimum $k-$power dominating set with $\ppt_k(G,\{u_1\})\leq n-4$ since $u_1$ will dominate $\{v, s\}$ in step 1,  and if necessary, $s$ will force $\{u_2, u_3\}$ in step 2 and $v$ will force $w$ in step 2.

			Therefore, $G$ is a connected graph on 5 vertices with $\Delta(G)=3.$ Also note that all connected graphs on 5 vertices with maximum degree 3 have $\ppt_k(G)=2$  (for $k\geq 2).$ This completes the proof.
		\end{proof}

		Next we consider graphs with $\ppt(G)=n-3$. We first characterize all trees with $\ppt(G)=n-3$.  then we characterize all graphs with  $\ppt(G)=n-3$ and $\gamma_{P}(G)\in \{2,3\}$. In Figure \ref{graphsnminus3}, we provide some graphs with $\ppt(G)=n-3$ and $\gamma_{P}(G)=1$, but  characterizing all such graphs is less tractable.

		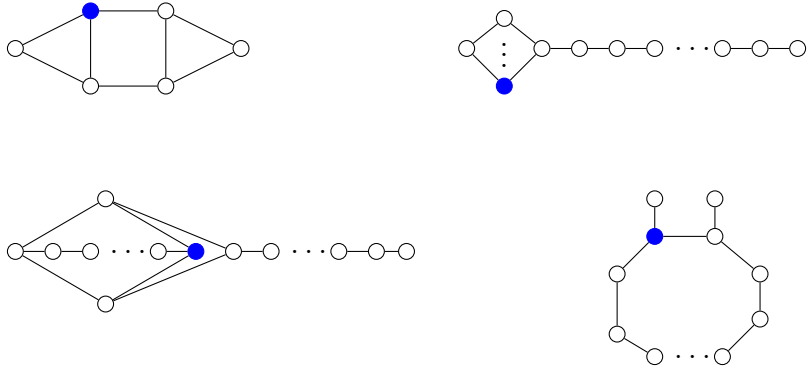
\begin{figure}[h!]
			\begin{center}
				\begin{tikzpicture}[scale=1.0]
				
				
				\bvertex (1) at (0,0) {};
				
				\vertex (2) at (1,0) {};
				
				\vertex (3) at (0,-1) {};
				
				\vertex (4) at (1,-1) {};
				
				\vertex (5) at (-1,-0.5) {};
				
				\vertex (6) at (2,-0.5) {};

				\draw(1) to (2);
				
				\draw(3) to (4);
				
				\draw(1) to (3);
				
				\draw(2) to (4);
				
				\draw(2) to (6);
				
				\draw(4) to (6);
				
				\draw(1) to (5);
				
				\draw(5) to (3);
				
				
				\vertex (7) at (5,-0.5) {};
				
				
				\vertex (9) at (6,-0.5) {};
				
\vertex (10) at (5.5,-0.1) {};
				
				\bvertex[label=above: $\vdots$] (11) at (5.5,-1) {};
				
				\vertex (12) at (6.5,-0.5) {};
				
				\vertex (13) at (7,-0.5) {};
				
				\vertex [label=right:$\ldots$](14) at (7.5,-0.5) {};
				
				\vertex (15) at (8.4,-0.5) {};
				
				\vertex (16) at (8.9,-0.5) {};
				
				\vertex (17) at (9.4,-0.5) {};

\draw(7) to (10);
				
\draw(10) to (9);
				
\draw(9) to (11);
				
\draw(11) to (7);		
				
\draw(9) to (12);
				
\draw(12) to (13);
				
\draw(13) to (14);
				
\draw(15) to (16);
				
\draw(16) to (17);
				
				

				
				\vertex  (18) at (-1,-3.2) {};
				
				\vertex  (19) at (-0.5,-3.2) {};
				
				\vertex  [label=right:$\ldots$](20) at (0,-3.2) {};
				
				\vertex  (21) at (0.9,-3.2) {};
				
				\bvertex  (22) at (1.4,-3.2) {};
				
				\vertex  (23) at (0.2,-2.5) {};
				
				\vertex  (24) at (0.2,-3.9) {};
				
				\vertex  (25) at (1.9,-3.2) {};
				
				\vertex  [label=right:$\ldots$](26) at (2.4,-3.2) {};
				
				\vertex  (27) at (3.3,-3.2) {};
				
				\vertex  (28) at (3.8,-3.2) {};
				
				\vertex  (29) at (4.2,-3.2) {};

				\draw(18) to (23);
				
				\draw(23) to (22);
				
				\draw(24) to (22);
				
				\draw(18) to (24);
				
				\draw(18) to (19);
				
				\draw(19) to (20);
				
				\draw(21) to (22);
				
				\draw(24) to (25);
				
				\draw(23) to (25);
				
				\draw(25) to (26);
				
				\draw(27) to (28);
				
				\draw(28) to (29);


				\vertex (30) at (7.5,-2.5) {};
				
				\vertex  (31) at (8.3,-2.5) {};
				
				\bvertex  (32) at (7.5,-3) {};
				
				\vertex  (33) at (8.3,-3) {};

				\vertex (34) at (7,-3.5) {};
				
				\vertex  (35) at (7,-4.3) {};
				
				\vertex  [label=right:$\ldots$](36) at (7.5,-4.6) {};
				
				\vertex  (37) at (8.4,-4.6) {};
				
				\vertex  (38) at (8.9,-4.1) {};
				
				\vertex  (39) at (8.9,-3.5) {};

				\draw(30) to (32);
				
				\draw(31) to (33);
				
				\draw(32) to (34);
				
				\draw(34) to (35);
				
				\draw(35) to (36);
				
				\draw(32) to (33);
				
				\draw(37) to (38);
				
				\draw(38) to (39);
				
				\draw (39) to (33);

				\end{tikzpicture}
				
				\caption{Graphs $G$ with $\ppt(G)=n-3$ and $\gamma_P(G)=1$. An efficient power dominating set in blue.}
				
				\label{graphsnminus3}
				
			\end{center}
		\end{figure}
		

\begin{prop}\label{Treeorderminus3} Let $T$ be a tree on $n$ vertices such that $\ppt(T)=n-3$. Then $T\in \{P_5, P_6, {\rm sp}(1,1,k)$ {\rm (for some} $k\geq 1)\}.$ 
			\end{prop}			
			
					\begin{proof} Suppose $T$ is a tree on $n$ vertices with $\ppt(T)=n-3.$ If $\Delta(T)\leq2,$ then $T$ must be a path, and by Proposition \ref{proppath}, $T=P_5$ or $T=P_6$. Suppose $\Delta(T)\geq 3$. From Lemma \ref{HHHHminpowerdomwithdegthree}, there exists a minimum power dominating set $S$ such that each vertex in $S$ has degree at least 3, so $|N[S]|\geq 4$ and $\ppt(T,S)\leq n-3.$ Since $\ppt(T)=n-3$ by assumption, then it must be the case that $\ppt(T,S)=n-3$. Thus, $|S|=1$ and $|N[S]|=4.$ \\\indent Let $S=\{s\}$ and $N[S]=\{s, u_1, u_2, u_3\}$. Note that the path $(u_i, s, u_j)$ (for $i\neq j)$ is the unique path from $u_i$ to $u_j$ (since $T$ is a tree), so the graph $T'=T-s$ has 3 connected components $T_1, T_2, T_3$ with $u_i\in T_i$. By Observation \ref{NeighborhoodZFS}, $\{u_1,u_2, u_3\}$ is a zero forcing set for $T'$, and it follows that $\{u_i\}$ is a zero forcing set of $T_i$. Since $T_i$ has zero forcing number 1, then $T_i$ is a path and $u_i$ is an endpoint of $T_i$ (\cite{row}). This gives that $T=sp(1,1,k)$ for some $k\geq 1.$ \end{proof}

		\begin{thm}\label{orderminus3} Let $G$ be a graph on $n$ vertices with $\ppt(G)=n-3$ and $\gamma_p(G)\in \{2,3\}$. Then $G\in  \{\overline{K_3}, \overline{K_2}\cup K_2, K_1\cup C_3, K_1\cup P_3, K_1\cup P_4, K_1\cup C_4, K_2\cup K_2\}$.
		\end{thm}

		\begin{proof} For any minimum power dominating set $S$ of $G$, $|S|\leq 3$ and $|N[S]|\leq 4$. Suppose $\Delta(G)\geq 3,$ and let $G_1$ be a connected component of $G$ containing a vertex of degree at least 3. By Lemma \ref{minpowerdomwithdegthree}, $G_1$ has a minimum power dominating set $S_1$ such that each $s\in S_1$ has degree at least 3. Let $S$ be a minimum power dominating of $G$ such that $S_1\subseteq S$. Since $|S|\in \{2,3\}$ and each $s\in S_1$ has degree at least 3, it follows that $|N[S]|\geq 5$, which is a contradiction. Thus, $\Delta(G)\leq 2$ and $G$ is the union of paths and cycles. Furthermore, $G$ has at least 2 connected components (since $\gamma_P(G)\neq 1$ then $G$ is not a path or cycle), and $G$ has at most 3 connected components (since $\gamma_P(G)\leq 3$).\\\indent Suppose $G$ has only two components, $G_1$ and $G_2$, so $\gamma_p(G)=2$. If $G_1$ is a path on at least 3 vertices or a cycle, then $G_2=K_1$ (since $|N[S]|\leq 4$) and $\ppt(G)=\ppt(G_1)=|G_1|-2$. By Proposition \ref{orderminus2}, $G_1 \in\{P_3, P_4, C_3, C_4\}$. Otherwise, $G=K_2\cup K_2$.\\\indent If $G$ has three components $G_1,G_2,G_3$, then $\gamma_p(G)=3$ and exactly one force is performed at each step. So, $G_2=G_3=K_1$, and $\ppt(G)=\ppt(G_1)=|G_1|-1$. By Proposition \ref{orderminus1}, $G_1\in\{K_1, K_2\}$.

		\end{proof}

		\noi {\bf Acknowledgements}\newline \noi The author would like to thank Dr. Leslie Hogben for her insight throughout this project, and Dr. Steve Butler for his assistance in the coding aspect of this project.

		

		\bibliographystyle{alpha}

	\end{document}